\documentclass[11pt,oneside,english]{article}
\usepackage{amsmath, amsthm, amssymb, amsfonts, enumerate, enumitem, xspace, mathtools}
\usepackage[T1]{fontenc}
\usepackage[latin9]{inputenc}
\usepackage{geometry}
\usepackage{amstext}
\usepackage{amsthm}
\usepackage{amsmath,amsfonts,amsthm,amssymb,mathrsfs,dsfont} % Math packages
\usepackage{mathtools}
\usepackage{graphicx}
\usepackage{xcolor}
\usepackage{bbm}
\usepackage[colorinlistoftodos,textsize=scriptsize,backgroundcolor=orange!20]{todonotes}
\usepackage[colorlinks=true, allcolors=blue]{hyperref}
\usepackage[capitalize,nameinlink]{cleveref}
\usepackage{verbatim}

\allowdisplaybreaks

\title{Resilience of the Rank of Random Matrices}

\author{Asaf Ferber \thanks{Department of Mathematics, University of California, Irvine. Email: asaff@uci.edu} \and
	Kyle Luh \thanks{Center of Mathematical Sciences and Applications, Harvard University. Email: kluh@cmsa.fas.harvard.edu} \and Gweneth McKinley \thanks{Department of Mathematics, MIT. Email: gweneth@mit.edu}  }

\date{\today}

\allowdisplaybreaks

\theoremstyle{plain}
\newtheorem{theorem}{Theorem}[section]
\newtheorem{lemma}[theorem]{Lemma}
\newtheorem{claim}[theorem]{Claim}

\newtheorem{definition}[theorem]{Definition}

\newcommand{\eps}{\varepsilon}

\newcommand{\LL}{\mathcal{L}}
\newcommand{\Z}{\mathbb{Z}}
\newcommand{\R}{\mathbb{R}}

\renewcommand{\a}{\mathbf{a}}

\newcommand{\Rka}{R_k^{\alpha}}

\DeclareMathOperator{\supp}{supp}

\DeclareMathOperator{\Bad}{\boldsymbol{B}}

\begin{document}

\global\long\def\R{\mathbb{R}}

\global\long\def\S{\mathcal{R}}

\global\long\def\Z{\mathbb{Z}}

\global\long\def\C{\mathbb{C}}

\global\long\def\Q{\mathbb{Q}}

\global\long\def\N{\mathbb{N}}

\global\long\def\P{\Pr}

\global\long\def\F{\mathbb{F}}

\global\long\def\U{\mathcal{U}}

\global\long\def\V{\mathcal{V}}

\global\long\def\E{\mathbb{E}}

\global\long\def\rk{\text{rank}}

\global\long\def\A{\mathcal{A}}

\global\long\def\L{\mathcal{L}}

\global\long\def\QQ{\mathcal{Q}}

\def \a {\alpha}
\def \b {\beta}
\def \g {\gamma}
\def \e {\varepsilon}
\def \d {\delta}
\def \l {\lambda}

\def \EE {\mathcal{E}}
\def \LL {\mathcal{L}}
\def \MM {\mathcal{M}}
\def \NN {\mathcal{N}}
\def \cZ {\mathcal{Z}}

	\maketitle
	\abstract{Let $M$ be an $n \times m$ matrix of independent Rademacher ($\pm 1$) random variables. It is well known that if $n \leq m$, then $M$ is of full rank with high probability.  We show that this property is resilient to adversarial changes to $M$. More precisely, if $m \geq n + n^{1-\eps/6}$, then even after changing the sign of $(1-\eps)m/2$ entries, $M$ is still of full rank with high probability. Note that this is asymptotically best possible as one can easily make any two rows proportional with at most $m/2$ changes. Moreover, this theorem gives an asymptotic solution to a slightly weakened version of a conjecture made by Van Vu in \cite{vu2008discrete}. 
	}

\section{Introduction}
Random discrete matrices, in particular $0/1$ and $\pm 1$ random matrices, have a distinguished history in random matrix theory.  They have applications in computer science, physics, and random graph theory, among others, and numerous investigations have been tailored to this class of random matrices \cite{bourgain2010singularity,kahn1995probability, komlos1967determinant,nguyen2013singularity, tao2007singularity,  tao2009inverse,   tikhomirov2018}.  Discrete random matrices are often of interest in their own right as they pose combinatorial questions that are vacuous or trivial for other models such as the gaussian ensembles (e.g. singularity and simpleness of spectrum).  For example, it is already non-trivial to show that a Bernoulli (0/1) random matrix is non-singular with probability $1-o(1)$ (this was firstly proved by Koml\'os in \cite{komlos1967determinant}).    
For an $n \times n$ Bernoulli random matrix $M_n$, it was a long standing conjecture that 
\[
p_n:=\P(M_n \text{ is singular}) = \left( \frac{1}{2} + o(1) \right)^n,
\]
which corresponds to the probability that any two rows or columns are identical. This problem has stimulated much activity \cite{kahn1995probability, tao2007singularity, bourgain2010singularity}, culminating in the recent resolution by Tikhomirov \cite{tikhomirov2018} of the above conjecture.

In this work, we examine another aspect of the singularity problem for discrete random matrices.  We will be concerned with robustness of the non-singularity, meaning how many changes to the entries of the matrix need to be performed to make a typical random matrix singular. This has been called the ``resilience'' of a random matrix with respect to singularity \cite{vu2008discrete}. Note that an $n\times n$ matrix is singular if and only if its rank is less than $n$. Therefore, we can extend the above notion for general matrices (not necessarily square) as follows: 

	\begin{definition}
		Given an $n\times m$ matrix $M$ with entries in $\{\pm1\}$, we denote by $Res(M)$ the minimum number of sign flips necessary in order to make $M$ of rank less than $n$.
	\end{definition}

Note that for every two $\pm 1$ vectors $\boldsymbol{a},\boldsymbol{b} \in \{\pm 1\}^m$ one can easily achieve either $\boldsymbol{a}=\boldsymbol{b}$ or $\boldsymbol{a}=-\boldsymbol{b}$ by changing at most $m/2$ entries; so in particular, for an $n\times m$ matrix $M$ we have the deterministic upper bound $$Res(M)\leq m/2.$$ 
Indeed, for the case $n=m$ it is conjectured by Vu that
\[
Res(M_n) = \left( \frac{1}{2} + o(1) \right) n
\]
with probability $1 - o(1)$ \cite[Conjecture 7.4]{vu2008discrete}.  
Note that by a a simple union bound, using any exponential upper bound on $p_n$, one can easily show that a.a.s. we have
	$$Res(M_n)\geq cn/\log n$$
for some appropriate choice of $c>0$. Surprisingly, no better lower bound is known.

In this paper we prove that for $m\geq (1+o(1))n$, the trivial upper bound $m/2$ is asymptotically tight. Before stating our main result we define the following notation: given $n,m\in \mathbb{N}$, we let $M_{n,m}$ be an $n\times m$ matrix with independent entries chosen uniformly from $\{\pm 1\}$. 

\begin{theorem} 
		\label{thm:resilience}
		For every $\varepsilon>0$ and $m\geq n+n^{1-\eps/6}$, a.a.s. we have
		$$Res(M_{n,m})\geq (1-\varepsilon)m/2.$$
	\end{theorem}

Our proof strategy roughly goes as follows: Consider an outcome $M$ of $M_{n,m}$. Note that if the rank of $M$ is less than $n$, then in particular, writing $m' = m - n^{1-\eps/6}$, there exists an $n\times m'$ submatrix $M'$ of $M$ with rank less than $n$. Moreover, as $M'$ is not of full rank, there exists $\boldsymbol{a}\in \mathbb{R}^n\setminus \{\boldsymbol{0}\}$ which lies in the \emph{left kernel of} $M'$ (that is, with $\boldsymbol{a}^TM'=\boldsymbol{0}$). Our main goal is to show that for each such $\boldsymbol{a}$ (if it exists), and for a randomly chosen $\boldsymbol{x}\in \{\pm 1\}^n$, the probability 
$$\rho(\boldsymbol{a}):=\Pr[\boldsymbol{a}^T\boldsymbol{x}=0]$$
is typically very small.

Next, observe that a vector $\boldsymbol{a}$ will be in the left kernel of $M$ if and only if it is in the left kernel of $M'$ and is also orthogonal to the remaining $n^{1-\eps/6}$ columns of $M$. Therefore, using the bound on $\rho(\boldsymbol{a})$ and the extra $n^{1-\eps/6}$ columns of $M$, we want to ``boost'' the probability and show that 
\begin{equation}\label{eq:main bound}
    \Pr[\exists \boldsymbol{a} \text{ such that } \boldsymbol{a}^TM=\boldsymbol{0}]=n^{-(1/2-o(1))m}.
\end{equation}

Note that since there are at most $\binom{nm}{(1/2-o(1))m}\approx n^{(1/2+o(1))m}$ many matrices that can be obtained from $M$ by changing $s\leq (1/2-o(1))m$ entries, and there are at most $2^{m}$ many choices for $M'$, using the above bound we complete the proof by a simple union bound (and of course, showing that the $o(1)$ terms in \eqref{eq:main bound} work in our favor). 

The main challenge is to prove \eqref{eq:main bound}, as it involves a union bound over all possible $\boldsymbol{a}\in \mathbb{R}^n$. In order to overcome this difficulty, we use some recently developed machinery introduced in \cite{ferber2019counting}. Roughly speaking, we embed the problem into a sufficiently large finite field $\F_p$. Then, as there are finitely many options for $\boldsymbol{a}\in\mathds{F}_p$ in the left kernel of $M$, we can use a counting argument from \cite{ferber2019counting} to bound the probability of encountering each possible kernel vector $\boldsymbol{a}$ according to the corresponding value of $\rho(\boldsymbol{a})$.

We mention that the approach of bounding $\rho(\boldsymbol{a})$ for possible null-vectors in the context of singularity is not new (see for example \cite{ kahn1995probability, nguyen2013singularity, rudelson2008LO,tao2009inverse,  vershynin2014symmetric}).  The novelty of our argument is that we utilize the methods in \cite{ferber2019counting} to obtain the bound \eqref{eq:main bound}.  Most of the previously used arguments yield exponential or polynomial probabilities which would only tolerate a sublinear number of modifications to the matrix.  Although it is possible to modify the previous arguments to generate super-exponential bounds, the exact constant of $1/2$ in \eqref{eq:main bound} seems to be difficult to achieve via other arguments. 

Lastly, we mention that the method in \cite{ferber2019counting} has already been successfully applied to a variety of combinatorial problems in random matrix theory \cite{ Rob, ferber2018singularity,  jain2019b, jain2019combinatorial, luh2019random}.

The remainder of this paper is organized as follows.  In \cref{sec:auxiliary}, we provide the necessary background to state the counting lemma from \cite{ferber2019counting}. %; we also provide the necessary proofs to modify and optimize the arguments in \cite{ferber2019counting} to obtain the constant $1/2$ in our main result.  %
In  \cref{sec:good-bad-vectors}, we provide a convenient interface to apply the counting lemma.  This is drawn from \cite{ferber2019counting} as well.  
Finally, in \cref{sec:resilience}, we provide the short proof of \cref{thm:resilience}.

\section{Auxiliary results} \label{sec:auxiliary}
Here we review some auxiliary results and introduce convenient notation to be used in the proof of our main result. 

\subsection{Hal\'asz inequality in $\F_p$}

Let $\boldsymbol{a}:= (a_1,\dotsc, a_n) \in (\Z \setminus \{0\})^{n}$ and let $\epsilon_1,\dotsc, \epsilon_n$ be independent and identically distributed (i.i.d.) Rademacher random variables; that is, each $\epsilon_i$ independently takes values $\pm 1$ with probability $1/2$ each. We define the largest atom probability $\rho(\boldsymbol{a})$ by
\[
  \rho(\boldsymbol{a}) := {\textstyle \sup_{x\in \Z}}\Pr\left(\epsilon_1 a_1 + \dotsb + \epsilon_n a_n = x\right).
\]

Similarly, if we are working over some finite field $\F_p$, let
\[
  \rho_{\F_p}(\boldsymbol{a}) := {\textstyle \sup_{x\in \F_p}}\Pr\left(\epsilon_1 a_1 + \dotsb + \epsilon_n a_n = x\right),
 \]
where, of course, the arithmetic is done over $\F_p$. 

Now, let $R_k(\boldsymbol{a})$ denote the number of solutions to $\pm a_{i_1} \pm a_{i_2} \dotsb \pm a_{i_{2k}} \equiv 0$, where repetitions are allowed in the choice of $i_1,\dots,i_{2k} \in [n]$.  
A classical theorem of Hal\'asz \cite{halasz1977estimates} gives an estimate on the atom probability based on $R_k(\boldsymbol{a})$. Here we need the following, slightly different version of this theorem, which can be applied to the finite field setting. 

\begin{theorem}[Hal\'asz's inequality over $\F_p$; Theorem 1.4 in \cite{ferber2019counting}] 
  \label{thm:halasz-fp}
  There exists an absolute constant $C$ such that the following holds for every odd prime $p$, integer $n$, and vector $\boldsymbol{a}:=(a_1,\dotsc, a_n) \in \F_p^{n}\setminus \{\boldsymbol{0}\}$. Suppose that an integer $k \ge 0$ and positive real $M$ satisfy $30M \leq |\supp(\boldsymbol{a})|$ and $80kM \leq n$. Then,
  \[
    \rho_{\F_p}(\boldsymbol{a})\leq \frac{1}{p}+\frac{CR_k(\boldsymbol{a})}{2^{2k} n^{2k} \cdot M^{1/2}} + e^{-M}.
  \]
\end{theorem}
For completeness, even though the proof is (more or less) identical to the original one by Hal\'asz, we include it in full in \cref{app:halasz}.

\subsection{Counting Lemma}

In this section we state a counting lemma from \cite{ferber2019counting} which plays a key role in our proof. First, we need the following definition: 

\begin{definition}
  Suppose that $\boldsymbol{a}\in \F_{p}^{n}$ for an integer $n$ and a prime $p$ and let $k \in \N$. For every $\alpha \in [0,1]$, we define $\Rka(\boldsymbol{a})$ to be the number of solutions to
  \[
    \pm a_{i_1}\pm a_{i_2}\dotsb \pm a_{i_{2k}}= 0 \mod p
  \]
  that satisfy $|\{i_1, \dotsc, i_{2k}\}| \ge (1+\alpha)k$.
\end{definition}

It is easily seen that $R_k(\boldsymbol{a})$ cannot be much larger than $\Rka(\boldsymbol{a})$. This is formalized in the following simple lemma, which is proved in \cite{ferber2019counting}.

\begin{lemma}
\label{lemma:R_k vs Rka}
For all $k,n\in\mathbb{N}$ with $k \le n/2$, and any prime $p$, vector $\boldsymbol{a} \in \F^n_p$, and $\alpha \in [0,1]$,
\[
  R_k(\boldsymbol{a})\leq  \Rka(\boldsymbol{a}) + \left(40 k^{1-\alpha}n^{1+\alpha}\right)^k.
\]
\end{lemma}

  \begin{proof}
  By definition, $R_k(\boldsymbol{a})$ is equal to $\Rka(\boldsymbol{a})$ plus the number of solutions to $\pm a_{i_1}\pm a_{i_2}\dotsb\pm a_{i_{2k}} = 0$ that satisfy $|\{i_1, \dotsc, i_{2k}\}| < (1+\alpha)k$. The latter quantity is bounded from above by the number of sequences $(i_1, \dotsc, i_{2k}) \in [n]^{2k}$ with at most $(1+\alpha)k$ distinct entries times $2^{2k}$, the number of choices for the $\pm$ signs. Thus
  \[
    R_k(\boldsymbol{a}) \leq \Rka(\boldsymbol{a}) + \binom{n}{(1+\alpha)k} \big((1+\alpha)k\big)^{2k}2^{2k} \leq \Rka(\boldsymbol{a}) +  \left(4e^{1+\alpha}k^{1-\alpha}n^{1+\alpha}\right)^k,
  \]
  where the final inequality follows from the well-known bound $\binom{a}{b} \le (ea/b)^b$. Finally, noting that $4e^{1+\alpha} \leq 4e^{2} \leq 40$ completes the proof.
\end{proof}

Given a vector $\boldsymbol{a}\in \F_p^n$ and a subset of coordinates $I\subseteq [n]$, we define $\boldsymbol{a}_I$ to be its restriction to the coordinates in $I$; that is, $\boldsymbol{a}_I=(a_i)_{i\in I}\in \F_p^I$. We write $\boldsymbol{b}\subseteq \boldsymbol{a}$ if there exists an $I\subseteq [n]$ for which $\boldsymbol{b}=\boldsymbol{a}_I$. For $\boldsymbol{b}\subseteq \boldsymbol{a}$ we let $|\boldsymbol{b}|$ be the size of the subset $I$ determining $\boldsymbol{b}$.

Now we are ready to state the counting lemma, and for the reader's convenience, we include the full (and relatively short) proof from \cite{ferber2019counting} in \cref{sec:pf-counting-thm}.	

\begin{theorem}[Theorem 1.7 in \cite{ferber2019counting}]
  \label{thm:counting-lemma}
  Let $p$ be a prime, let $k, n \in \N$, $s\in [n]$, $t\in [p]$, and let $\alpha \in (0,1)$. Denoting
  \[
    \Bad_{k,s,\geq t}^{\alpha}(n):= \left\{\boldsymbol{a} \in \F_{p}^{n} : R^{\alpha}_k(\boldsymbol{b})\geq t\cdot \frac{2^{2k} \cdot |\boldsymbol{b}|^{2k}}{p} \text{ for every } \boldsymbol{b}\subseteq \boldsymbol{a} \text{ with } |\boldsymbol{b}|\geq s\right\},
  \]
  we have
  \[
    |\Bad_{k,s,\geq t}^{\alpha}(n)| \leq \left(\frac{s}{n}\right)^{2k-1} (\alpha t)^{s-n} p^n.
  \]
\end{theorem}
	
\subsection{``Good'' and ``bad'' vectors} 
\label{sec:good-bad-vectors}

The purpose of this section is to formulate easy-to-use versions of Hal\'asz's inequality (\cref{thm:halasz-fp}) and our counting theorem (\cref{thm:counting-lemma}). This follows \cite{ferber2019counting} closely, but requires a more delicate choice of parameters as we need to achieve the bound in \eqref{eq:main bound} (and crucially, the constant $1/2$ in the exponent).  We shall partition $\F_p^n$ into ``good'' and ``bad'' vectors. We shall then show that, on the one hand, every ``good'' vector $\boldsymbol{a}$ has a small $\rho(\boldsymbol{a})$ and that, on the other hand, there are relatively few ``bad'' vectors.\footnote{In fact, we shall only show that there are relatively few ``bad'' vectors that have some number of nonzero coordinates. The number of remaining vectors (ones with very small support) is so small that even a very crude estimate will suffice for our needs.} The formal statements now follow. In order to simplify the notation, we suppress the implicit dependence of the defined notions on $n$, $k$, $p$, and $\alpha$.

\begin{definition}
  \label{def:Ht-rrm}
  Let $p$ be a prime, let $n,k\in \mathbb{N}$, and let $\alpha \in (0,1)$. For any $t>0$, define the set $\boldsymbol{H}_t$ of \emph{$t$-good} vectors by
  \[
    \boldsymbol{H}_t:= \left\{\boldsymbol{a}\in \F_p^n: \exists \boldsymbol{b}\subseteq \boldsymbol{a} \text{ with }|\supp(\boldsymbol{b})|\geq n^{1- \eps/2} \text{ and }\Rka(\boldsymbol{b})\leq t\cdot \frac{2^{2k} \cdot |\boldsymbol{b}|^{2k}}{p}\right\}.
  \]
  The \emph{goodness} of a vector $\boldsymbol{a} \in \F_p^n$, denoted by $h(\boldsymbol{a})$, will be the smallest $t$ such that $\boldsymbol{a} \in \boldsymbol{H}_t$. In other words
  \[
    h(\boldsymbol{a}) = \min\left\{\frac{p  \cdot \Rka(\boldsymbol{b})}{2^{2k} \cdot |\boldsymbol{b}|^{2k}} : \boldsymbol{b} \subseteq \boldsymbol{a} \text{ and } |\supp(\boldsymbol{b})| \ge n^{1-\eps/2}\right\}.
  \]
\end{definition}

Note that if a vector $\boldsymbol{a} \in \F_p^n$ has fewer than $n^{1 - \eps/2}$ nonzero coordinates, then it cannot be $t$-good for any $t$ and thus $h(\boldsymbol{a}) = \infty$. On the other hand, trivially $\Rka(\boldsymbol{b}) \le 2^{2k} \cdot |\boldsymbol{b}|^{2k}$ for every vector $\boldsymbol{b}$, as there are $2^{2k}|\boldsymbol{b}|^{2k}$ total possible choices of a sequence $\pm b_{i_1}\pm b_{i_2}\pm \cdots \pm b_{i_{2k}}$. Thus every $\boldsymbol{a} \in \F_p^n$ with at least $n^{1-\eps/2}$ nonzero coordinates must be $p$-good, that is, $h(\boldsymbol{a}) \le p$ for each such $\boldsymbol{a}$.

Having formalized the notion of a ``good" vector, we are now ready to state and prove two corollaries of \cref{thm:halasz-fp,thm:counting-lemma} that lie at the heart of our approach. (Note: the particular choice of parameters in \cref{lem:smallballdreg} is made for convenience in a later application.)

\begin{lemma}
  \label{lem:smallballdreg}
  Let $\boldsymbol{a} \in \boldsymbol{H}_t$, let $\alpha \in (0,1)$, and let $\eps<1/100$. Suppose that $p=\Theta(2^{n^{\eps/3}})$ is a prime, $t \geq n$, and $k = \Theta(n^{\eps/3})$. Then for sufficiently large $n$ we have
  \[
    \rho_{\F_p}(\boldsymbol{a}) \le \frac{Ct}{p n^{\frac{1}{2}(1 - 5\eps/6})},
  \]
  where $C = C(\alpha, \eps)$ is a constant depending only on $\alpha$ and $\eps$. 
\end{lemma}
\begin{proof}
  As $\boldsymbol{a} \in \boldsymbol{H}_t$, we can find a subvector $\boldsymbol{b}$ of $\boldsymbol{a}$ such that $|\supp(\boldsymbol{b})| \ge n^{1-\eps/2}$ and $\Rka(\boldsymbol{b}) \le t \cdot 2^{2k} \cdot |\boldsymbol{b}|^{2k}/p$. Set $M = \lfloor n^{1-\eps/2} / (80k) \rfloor = \Theta(n^{1-5\eps/6})$ so that
  \[
    \max\{30M, 80Mk\} = 80Mk \le n^{1-\eps/2} \le |\supp(\boldsymbol{b})| \le |\boldsymbol{b}|.
  \]
  Thus we may apply \cref{thm:halasz-fp} to obtain, for some absolute constant $C_0$, 
  $$\rho_{\F_{p}}(\boldsymbol{b})  \le \frac{1}{p}+\frac{C_0R_{k}(\boldsymbol{b})}{2^{2k} \cdot |\boldsymbol{b}|^{2k} \cdot M^{1/2}}+e^{-M}. $$
  Now, using \cref{lemma:R_k vs Rka} we can upper bound the right hand side by
  \begin{align*}
    \rho_{\F_{p}}(\boldsymbol{b}) & \leq \frac{1}{p}+\frac{C_0R_{k}^{\alpha}(\boldsymbol{b})+C_0\left(40k^{1-\alpha} |\boldsymbol{b}|^{1+\alpha}\right)^k}{2^{2k} \cdot |\boldsymbol{b}|^{2k} \cdot M^{1/2}}+e^{-M}\\
                                  & \leq \frac{1}{p}+\frac{C_0t \cdot 2^{2k} \cdot |\boldsymbol{b}|^{2k}/p+C_0\left(40k^{1-\alpha} |\boldsymbol{b}|^{1+\alpha}\right)^k}{2^{2k} \cdot |\boldsymbol{b}|^{2k} \cdot M^{1/2}}+e^{-M}\\
                                  & = \frac{1}{p} \left(1 + \frac{C_0t}{M^{1/2}} +C_0\left( 10 (k / |\boldsymbol{b}|)^{1-\alpha}\right)^k \cdot \frac{p}{M^{1/2}}\right) + e^{-M}.
  \end{align*}
Now we wish to show that, with the parameter assignments above, the dominant term in this sum is $\frac{C_0 t}{p M^{1/2}}$. To this end, we bound each of the other terms as follows. First, 

$$e^{-M} = e^{-\Theta(n^{1-5\eps/6})} = o(2^{-n^{\eps/3}}) = o\left(\frac{1}{p}\right).$$
(here we use the upper bound assumption on $\eps$.) Second,  

\begin{align*}C_0\left( 10 (k / |\boldsymbol{b}|)^{1-\alpha}\right)^k \cdot \frac{p}{M^{1/2}}&\leq C_0\left(10\left(n^{\eps/3-(1-\eps/2)}\right)^{1-\alpha}\right)^k \cdot p\\ 
&= \left(n^{-\Theta(1)}\right)^{\Theta(n^{\eps/3})} \cdot p\\
&= 2^{-\Theta(n^{\eps/3}\log n)}\cdot \Theta(2^{n^{\eps/3}})\\
&=o(1).
\end{align*}
And last, we observe that, as $t\geq n$, 
$$
\frac{C_0t}{M^{1/2}}\geq \frac{n}{\Theta(n^{\frac{1}{2}(1-5\eps/6}))} = \omega(1).
$$

Therefore the dominant term in the sum above is indeed $\frac{C_0t}{p M^{1/2}}$; then, choosing the constant $C=C(\alpha,\eps)>C_0$ sufficiently large, we obtain
$$
    \rho_{\F_{p}}(\boldsymbol{b}) \leq  \frac{C t}{p M^{1/2}}
    \leq \frac{C t}{p n^{\frac{1}{2}(1-5\eps/6)}}
$$
as desired. (Note: in the last step, we have incorporated the implicit constant in $M=\Theta(n^{1-5\eps/6})$ into the constant $C$.)
\end{proof}

\begin{lemma}
  \label{lemma:counting-bad}
  For every integer $n$ and real $t \ge n$,
  \[
    \left|\left\{\boldsymbol{a} \in \F_p^n : |\supp(\boldsymbol{a})| \ge n^{1 - \eps/2} \text{ and } \boldsymbol{a} \not\in \boldsymbol{H}_t\right\}\right| \le 2^n\left(\frac{p}{\alpha t}\right)^n \cdot t^{n^{1- \eps/2}}.
  \]
\end{lemma}
\begin{proof}
  We may assume that $t \le p$, as otherwise the left-hand side above is zero; see the comment below \cref{def:Ht-rrm}. Let us now fix an $S \subseteq [n]$ with $|S| \ge n^{1- \eps/2}$ and count only vectors $\boldsymbol{a}$ with $\supp(\boldsymbol{a}) = S$. Since $\boldsymbol{a} \not\in \boldsymbol{H}_t$, the restriction $\boldsymbol{a}_S$ of $\boldsymbol{a}$ to the set $S$ must be contained in the set $\boldsymbol{B}_{k,n^{1- \eps/2},\geq t}^\alpha(|S|)$. Hence, \cref{thm:counting-lemma} implies that the number of choices for $\boldsymbol{a}_S$ is at most
  \[
    \left(\frac{n^{1-\eps/2}}{|S|}\right)^{2k-1}(\alpha t)^{n^{1-\eps/2}-|S|}p^{|S|} \leq \left(\frac{p}{\alpha t}\right)^n t^{n^{1 - \eps/2}},
  \]
  where the second inequality follows as $n^{1-\eps/2}\leq |S|\leq n$ and $\alpha t \le t \le p$. Since $\boldsymbol{a}_S$ completely determines $\boldsymbol{a}$, we obtain the desired conclusion by summing the above bound over all sets $S$.
\end{proof}
	
\section{Proof of \cref{thm:resilience}} \label{sec:resilience}
In this section we gradually construct the entire proof of \cref{thm:resilience}.	
	
For convenience, we introduce some notation to indicate the distance of two Rademacher matrices.

\begin{definition}
	For two matrices $n\times m$ matrices $M, M'$ we let $d(M, M')$ denote the number of entries where $M$ and $M'$ differ.
\end{definition}

With this definition in hand, \cref{thm:resilience} can be stated as follows: 

\begin{theorem}
  For every $\varepsilon>0$ and $m\geq n+n^{1-\varepsilon/6}$, a.a.s. we have $rank(M')=n$ for all $n\times m$, $\pm 1$ matrices $M'$ with $d(M_{n,m},M')\leq (1-\varepsilon)m/2$. 
\end{theorem}

First, we will prove \cref{thm:resilience} under the assumption that $m=\omega(n)$. 

\subsection{Proof of  \cref{thm:resilience} under the assumption $m=\omega(n)$}\label{subsec:bigm}

Let $\varepsilon>0$ be any fixed constant, and let $m\geq C(\eps)n$, where $C(\varepsilon)$ is a sufficiently large constant. We wish to show that a.a.s., $M=M_{n,m}$ is such that every $n\times m$ matrix $M'$ with $d(M,M')\leq (1-\varepsilon)m/2$ has rank $n$. 

In order to do so, let us take (say) $p=3$ and work over $\F_3$. Observe that if the above statement holds over $\F_3$ then it trivially holds over $\mathbb{Z}$. 

Let $\boldsymbol{a}\in \F_3^n\setminus \{\textbf{0}\}$, and note that for a randomly chosen $\boldsymbol{x}\in \{\pm 1\}^n$ we have
$$\Pr[\boldsymbol{a}^T \boldsymbol{x}=0\,]\leq \frac 12.$$ 
Therefore, as the columns of $M$ are independent, it follows that the random variable $X_{\boldsymbol{a}}=$ ``the number of zeroes in $\boldsymbol{a}^T M$'' is stochastically dominated by $\text{Bin}(m,\frac 12)$. Hence, by Chernoff's bound, we obtain that 

$$\Pr\left[X_{\boldsymbol{a}}\geq (1+\varepsilon)m/2\right]\leq e^{-C_1m}$$ 
for some $C_1$ that depends on $\varepsilon$. By applying the union bound over all $\boldsymbol{a}\in \F_3^n\setminus \{\boldsymbol{0}\}$ we obtain that 
$$\Pr\big[\exists \boldsymbol{a}\in \F_3^n\setminus \{\boldsymbol{0}\} \text{ with } X_{\boldsymbol{a}}\geq (1+\varepsilon)m/2\big]\leq 3^ne^{-C_1m}=o(1),$$ 
where the last inequality follows from the fact that $m\geq C(\eps)n$ and $C(\eps)$ is sufficiently large. 

Thus $M$ is typically such that in \emph{every} non-zero linear combination of its rows, there are less than $(1+\varepsilon)m/2$ many zeroes. In particular, since by changing at most $(1-\varepsilon)m/2$ many entries one can affect at most $(1-\varepsilon)m/2$ columns, it follows that for all $M'$ with $d(M,M')\leq (1-\varepsilon)m/2$, no non-trivial combination of the rows of $M'$ is the $0$ vector. In particular, every such $M'$ is of rank $n$. This completes the proof for this case.  \hfill \qed

\subsection{Proof of  \cref{thm:resilience} under the assumption $m=O(n)$}

In what follows we always assume that $m = O(n)$. Therefore, whenever convenient, in appropriate asymptotic formulas we may switch between $m$ and $n$ without further explanation. This case is more involved than the case $m=\omega(n)$ and it will be further divided into a few subcases. From now on, we fix $p$ to be some prime $p=\Theta(2^{n^{\eps/3}})$, and concretely, we write $m\leq C(\eps)n$ for some constant $C(\eps)$. 

Now, write $m'= m-n^{1-\eps/6}$ (the width of the matrix under consideration minus the $n^{1-\eps/6}$ ``extra'' columns). In the following two subsections, we will show that with high probability, for every $\boldsymbol{a}	\in \F_p^n$, if $\boldsymbol{a}^T M'=\boldsymbol{0}$ for some $n\times m'$ matrix $M'$ with $d(M', M_{n,m'})\leq (1-\eps)m/2$, then $\boldsymbol{a}$ has ``many'' nonzero entries, and is ``pseudorandom'' in some sense (\cref{lem:getting rid of small dependencies,larger dependencies}). From here, we can apply the Halasz inequality (in the form of 
\cref{lem:smallballdreg}) almost directly, using the fact that there are $m-m'\geq n^{1-\eps/6}$ extra columns, to conclude that for any such $\boldsymbol{a}$, the probability that $\boldsymbol{a}^T M_{n,m} = \boldsymbol{0}$ is small.

\subsubsection{Eliminating Small Linear Dependencies}

First, we wish to show that if $\boldsymbol{a}^TM'=\boldsymbol{0}$ (over $\F_p$) for some $M'$ with $d(M_{n,m'},M')\leq (1-\varepsilon)m/2$, then $\boldsymbol{a}$ has ``many'' non-zero entries (assuming $\boldsymbol{a}\neq \boldsymbol{0}$ of course). 

	\begin{lemma} \label{lem:getting rid of small dependencies}
		Let $\varepsilon>0$, let $p=\Theta(2^{n^{\eps/3}})$ be a prime, and let $n + n^{1-\eps/6} \leq m \leq C(\eps) n$. Write $m' = m-n^{1-\eps/6}$. Then, working in $\F_p$, the probability there exists a matrix $M'$ with $d(M',M_{n,m'})\leq (1-\varepsilon)m/2$ and a nonzero vector $\boldsymbol{a}\in \F_p^n$ with $|\text{supp}(\boldsymbol{a})|\leq n^{1-\varepsilon/2}$ and with $\boldsymbol{a}^TM'=\boldsymbol{0}$ is at most $2^{-\Theta(n)}$.
	\end{lemma}
	
	\begin{proof}
		Given a vector $\boldsymbol{a}\in\F_p^n$, we let $\ell:=|\text{supp}(\boldsymbol{a})|$. Note that for any $\boldsymbol{a}\neq \boldsymbol{0}$ and a uniformly chosen vector $\boldsymbol{x}\in \{\pm 1\}^n$ we trivially have 
		$$\Pr[\boldsymbol{a}^T \boldsymbol{x}=0]\leq \frac 12.$$
		Moreover, as we are only allowed to change at most $(1-\varepsilon)m/2$ coordinates of $M_{n,m'}$, it follows that at most $(1-\varepsilon)m/2$ entries of $\boldsymbol{a}^TM_{n,m'}$ can be altered. In particular, if there exists a vector $\boldsymbol{a}$ for which $ \boldsymbol{a}^T M'=\boldsymbol{0}$, where $d(M_{n,m'},M')\leq (1-\varepsilon)m/2$, then this implies that $\boldsymbol{a}^TM_{n, m'}$ already contained at least $m'-(1-\eps)m/2 = (1+\varepsilon-o(1))m/2$ zero entries. 
		
		Now, since the random variable counting the number of $0$ entries is stochastically dominated by $\text{Bin}(n,\frac{1}{2})$, by Chernoff's bound we obtain that for a given $\boldsymbol{a}\neq \boldsymbol{0}$, the probability to have at least $(1+\varepsilon-o(1))m/2$ zeroes in $\boldsymbol{a}^TM_{n, m'}$ is at most $2^{-c(\varepsilon)m}$, where $c(\varepsilon)$ is some constant depending only on $\varepsilon$. Thus the probability that for a given nonzero vector $\boldsymbol{a}$, there exists some $M'$ with $d(M',M_{n,m'})\leq (1-\eps)m/2$ and  $\boldsymbol{a}^TM'=\boldsymbol{0}$ is at most $2^{-c(\varepsilon)m}$.
		
		All in all, by applying the union bound over all $\boldsymbol{a}\neq \boldsymbol{0}$ with $\ell\leq n^{1-\eps/2}$ nonzero entries, the probability that we are seeking to bound is at most
		$$\sum_{\ell =1}^{ n^{1-\varepsilon/2}}\binom{n}{\ell}p^\ell2^{-c(\varepsilon)m}\leq \sum_{\ell =1}^{ n^{1-\varepsilon/2}}2^{\ell\log n+\ell n^{\varepsilon/3}-c(\varepsilon)m}=2^{-\Theta(n)}$$
		where the last equality holds due to the assumption $\ell\leq n^{1-\varepsilon/2}$.
	\end{proof}

\subsubsection{Eliminating ``bad'' vectors}

We now show that, almost surely, any vector $\boldsymbol{a}$ with many non-zero entries and with $\boldsymbol{a}^TM'=\boldsymbol{0}$ for some $M'$ with $d(M',M_{n,m'})\leq (1-\eps)m/2$ will be ``good'' or ``unstructured''.

	\begin{lemma}
		\label{larger dependencies}
		Let $\eps>0$, let  $p=\Theta(2^{n^{\varepsilon/3}})$ be a prime% and let $k=n^{\eps/3}$.
		, and let $n+n^{1-\eps/6}\leq m\leq C(\eps)n$. Write $m'= m-n^{1-\eps/6}$. Then, working in $\F_p$, the probability that there exists a matrix $M'$ with $d(M',M_{n,m'})\leq (1-\eps)m/2$ and a vector $\boldsymbol{a}\in \F_p^n\setminus \boldsymbol{H}_n$ with at least $n^{1-\eps/2}$ non-zero entries such that $\boldsymbol{a}^T M' =\boldsymbol{0}$ is at most $2^{-\Theta(n\log n)}$.
	\end{lemma}

	\begin{proof} 
		Our first step is to take a union bound over choices of $\boldsymbol{a}$; we wish to bound the quantity
		\begin{equation}\label{londerDependencyBound}
	 \sum_{\substack{\boldsymbol{a}\in \F_p^n\setminus \boldsymbol{H}_n \\ |\supp(\boldsymbol{a})|\geq\, n^{1-\eps/2}}} \textstyle\Pr[\exists M' \text{ with } d(M', M_{n,m'})\leq (1-\eps)m/2 \text{ and } \boldsymbol{a}^T M' =\boldsymbol{0}].
		\end{equation}
		Now we use the sets $\boldsymbol{H}_t$ to divide the vectors $\boldsymbol{a}$ into different classes. As observed after \cref{def:Ht-rrm}, every $\boldsymbol{a}\in\F_p^n$ with at least $n^{1-\eps/2}$ nonzero entries is in $\boldsymbol{H}_t$ for some $t\leq p$. Moreover, notice that $\boldsymbol{H}_t\subseteq \boldsymbol{H}_{t+1}$ for any $t > 0$. So we can write $\F_p^n\setminus \boldsymbol{H}_n$ as a union $\bigcup_{n+1\leq t\leq p} \boldsymbol{H}_t \setminus \boldsymbol{H}_{t-1}$. Therefore, taking a union bound over integers $t>n$, the probability \cref{londerDependencyBound} that we are trying to bound is at most
		\begin{equation*}%\label{londerDependencyBound}
		\sum_{t= n+1}^{p}\left( \sum_{\boldsymbol{a}\in \boldsymbol{H}_t\setminus \boldsymbol{H}_{t-1}} \textstyle\Pr[\exists M' \text{ with } d(M', M_{n,m'})\leq (1-\eps)m/2 \text{ and } \boldsymbol{a}^T M' = \boldsymbol{0}]\right).
		\end{equation*}
	
		Now, we take another union bound, this time over the possible edits to the matrix; by changing at most $(1-\varepsilon)m/2$ entries, an adversary can form
		\begin{equation*}\label{eq:UnionBdEdits}
		\sum_{i=0}^{(1-\eps)m/2}\binom{nm'}{i}\leq \left(\frac{2 e n}{1 - \eps}\right)^{(1-\varepsilon)m/2}= 2^{(1-\varepsilon+o(1))\frac{m}{ 2 }\log n}
		\end{equation*} many $n\times m'$ matrices. Thus \cref{londerDependencyBound} is at most
		\[
		\sum_{t= n+1}^{p} \left(\sum_{\boldsymbol{a}\in \boldsymbol{H}_t\setminus \boldsymbol{H}_{t-1}} \textstyle2^{(1-\varepsilon+o(1))\frac{m}{2 }\log n}\cdot \Pr[\boldsymbol{a}^TM_{n,m'} = \boldsymbol{0}]\right).
		\]
		(Note: this is possible because, by conditioning on the locations of the entries edited, each altered matrix $M'$ is distributed identically to $M_{n,m'}$.)

		We now wish to bound the probability that $\boldsymbol{a}^T M_{n,m'}  = \boldsymbol{0}$ for any fixed $\boldsymbol{a}\in \boldsymbol{H}_t \setminus \boldsymbol{H}_{t-1}$. By \cref{lem:smallballdreg} (as $\boldsymbol{a}\in \boldsymbol{H}_t$), and by the independence of the columns in $M_{n,m'}$, this probability is at most $\left(\frac{Ct}{p n^{\frac{1}{2}(1-5\eps/6)}}\right)^{m'}$. Therefore, \cref{londerDependencyBound} is at most
		\[
		\sum_{t=n+1}^{p} \left(\sum_{\boldsymbol{a}\in \boldsymbol{H}_t \setminus \boldsymbol{H}_{t-1}} 2^{(1-\varepsilon+o(1))\frac{m}{2}\log n}\cdot \left(\frac{Ct}{p n^{\frac{1}{2}(1-5\eps/6)}}\right)^{m'}\right).
		\]
		We now bound the number of vectors $\boldsymbol{a}$ in each $\boldsymbol{H}_t \setminus \boldsymbol{H}_{t-1}$. By definition, $\boldsymbol{H}_t \setminus \boldsymbol{H}_{t-1} \subset \F_p^n\setminus \boldsymbol{H}_{t-1}$, and by \cref{lemma:counting-bad}, the size of $\F_p^n\setminus \boldsymbol{H}_{t-1}$ is bounded above by
		\[\left(\frac{2p}{\alpha t}\right)^n \cdot t^{n^{1- \eps/2}},\]
		where $\alpha\in(0,1)$ is any fixed constant (note that the constant $C$ above depends on $\alpha$). Thus \cref{londerDependencyBound} is bounded by the following explicit expression:
		\begin{align*}
		&\sum_{t=n+1}^{p} \left(\frac{2p}{\alpha t}\right)^n \cdot t^{n^{1- \eps/2}} \cdot  2^{(1-\varepsilon+o(1))\frac{m}{2}\log n}\cdot \left(\frac{Ct}{p n^{\frac{1}{2}(1-5\eps/6)}}\right)^{m'}\\
		=\ & 2^{(1-\varepsilon+o(1))\frac{m}{2}\log n}\cdot n^{-(1-5\eps/6) \frac{m'}{2}} \cdot \left(\frac{2}{\alpha} \right)^{n}\cdot C^{m'}\sum_{t= n+1}^{p} \left(\frac{t}{p}\right)^{m'-n}t^{n^{1-\eps/2}}.
		\end{align*}
		Now, bounding each piece separately, and recalling that $m'\geq n$,
		\[
		\left(\frac{2}{\alpha}\right)^n C^{m'}= 2^{O(n)},\]
		\[\sum_{t= n+1}^{p} \left(\frac{t}{p}\right)^{m'-n}t^{n^{1-\eps/2}} \leq p\cdot1\cdot p^{n^{1-\eps/2}} = 2^{n^{\eps/3}}\cdot 2^{n^{\eps/3}\cdot n^{1-\eps/2}} = 2^{o(n)},
		\]
		\[
		2^{(1-\varepsilon+o(1))\frac{m}{2}\log n}\cdot n^{-(1-5\eps/6) \frac{m'}{2}} = 2^{-(1-o(1))\frac{\eps}{12}{\cdot m\log{n}}},
		\]
		where in the last equality, we use the fact that $m'= m-n^{1-\eps/6}= (1-o(1))m$. Thus in total, \cref{londerDependencyBound} is at most
		\[
		2^{(-\eps/12+o(1))m\log{n}} = 2^{-\Theta(n\log n)}.
		\]
		This completes the proof of the lemma.
	\end{proof}
	
	\subsubsection{Completing the proof}
	
   Given the assumption $m \leq C(\eps)n$, we will in fact prove something slightly stronger, namely that Theorem \ref{thm:resilience} holds over $\F_p$ for an appropriate choice of $p$. We wish to bound the probability that there exists some nonzero vector $\boldsymbol{a}\in\F_p^n$ with $\boldsymbol{a}^T M_{n,m} = \boldsymbol{0}$, even after at most $(1-\eps)m/2$ edits. Let $p=\Theta(2^{n^{\eps/3}})$ be prime. We begin by dividing into ``structured'' and ``unstructured'' vectors; for brevity, given a nonzero vector $\boldsymbol{a}$ and matrix $M$, we denote by $\mathcal{E}(\boldsymbol{a},M)$ the event that there exists a matrix $M'$ with $d(M',M)\leq (1-\eps)m/2$ and $\boldsymbol{a}^TM'=\boldsymbol{0}$.
  \begin{align}
		\textstyle\Pr[\exists \,\boldsymbol{a}\in \F_p^n \text{ with } &\mathcal{E}(\boldsymbol{a},M_{n,m})]\notag
		\\
		\textstyle\leq& 
		\Pr[\exists \,\boldsymbol{a}\in\boldsymbol{H}_n \text{ with } \mathcal{E}(\boldsymbol{a},M_{n,m})]\label{eq:FirstTerm}\\
		\textstyle
		+&\Pr[\exists \,\boldsymbol{a}\in\F_p^n\setminus\boldsymbol{H}_n \text{ with } \mathcal{E}(\boldsymbol{a},M_{n,m})],\label{eq:SecondTerm}
   \end{align}
		where $\boldsymbol{H}_{n}$ is the set of `good' or `unstructured' vectors defined in \cref{sec:good-bad-vectors}. The first summand (\ref{eq:FirstTerm}) is bounded as follows: first, take a union bound over possible edits to $M_{n,m}$. There are
		\[
		\sum_{i=0}^{(1-\eps)m/2}\binom{n m}{i} = 2^{(1-\eps+o(1))\frac{m}{2}\log n}
		\]
		possible choices for $M'$. Thus, for the first term (\ref{eq:FirstTerm}), we obtain a bound of
		\[
		2^{(1-\eps+o(1))\frac{m}{2}\log n}\textstyle\cdot \Pr[\exists\, \boldsymbol{a}\in \boldsymbol{H}_n \text{ with } \boldsymbol{a}^TM_{n,m} = \boldsymbol{0}].
		\]
		(As in the proof of \cref{larger dependencies}, this is possible because, by conditioning on the locations of the entries edited, each $M'$ is distributed identically to $M_{n,m}$.) And for $\boldsymbol{a}\in \boldsymbol{H}_{n}$, and $\boldsymbol{x}\in\{\pm 1\}^n$ chosen uniformly at random, \cref{lem:smallballdreg} gives
		\[
		\Pr\left[\boldsymbol{a}^T\boldsymbol{x}=0\right] \leq \frac{C n}{p n^{(1/2- 5\eps/12)}} < \frac{n}{p}.
		\]
		So for $M_{n,m}$ with $m\geq n+n^{1-\eps/6}$ columns, the probability of having $\boldsymbol{a}^TM_{n,m} = \boldsymbol{0}$ is at most $\left(\frac{n}{p}\right)^{n+n^{1-\eps/6}}$. Therefore, as there are at most $p^n$ vectors $\boldsymbol{a}\in \boldsymbol{H}_n$, and as $m\leq C(\eps)\cdot n$, the first summand (\ref{eq:FirstTerm}) is bounded by
		\begin{align*}
		2^{(1-\eps+o(1))\frac{m}{2}\log n} \cdot p^n \left(\frac{n}{p}\right)^{n+n^{1-\eps/6}}
		& =     2^{(1-\eps+o(1))\frac{m}{2}\log n} \cdot p^{-n^{1-\eps/6}} n^{n+n^{1-\eps/6}}\\
		& = 2^{O(n\log n)}\cdot 2^{-n^{\eps/3} n^{1-\eps/6}}\\
		& = 2^{-\Theta(n^{1+\eps/6})}.
		\end{align*}
		
		\noindent Now we bound the second summand (\ref{eq:SecondTerm}). We begin by restricting to the first $m' = m-n^{1-\eps/6}$ columns of $M_{n,m}$. This gives a strictly larger probability, as it is more likely that there is a linear dependency among the rows of a matrix when we restrict to only a subset of its columns. So (\ref{eq:SecondTerm}) is bounded above by
		\begin{align*}
		\textstyle\Pr[\exists \,\boldsymbol{a}\in\F_p^n&\setminus\boldsymbol{H}_n \text{ with } \mathcal{E}(\boldsymbol{a},M_{n,m'})]\\
		\leq &\Pr[\exists \,\boldsymbol{a}\in\F_p^n\setminus\boldsymbol{H}_n \text{ with } |\supp(a)|\geq n^{1-\eps/2} \text{ and } \mathcal{E}(\boldsymbol{a},M_{n,m'})]\\
		+ &\Pr[\exists \,\boldsymbol{a}\in\F_p^n\setminus\boldsymbol{H}_n \text{ with } |\supp(a)|< n^{1-\eps/2} \text{ and } \mathcal{E}(\boldsymbol{a},M_{n,m'})]
		\end{align*}
		And these are respectively the precise probabilities bounded in Lemmas \ref{larger dependencies} and \ref{lem:getting rid of small dependencies}. Therefore this is at most
		\[
		2^{-\Theta(n\log n)} + 2^{-\Theta(n)}.
		\]
		Thus in total, the probability that there exists a nonzero vector $\boldsymbol{a}\in\F_p^n$ with $\boldsymbol{a}^T M_{n,m} = \boldsymbol{0}$, even after at most $(1-\eps)m/2$ edits is at most
		\begin{align*}
 2^{-\Theta(n^{1+\eps/6})}    + 2^{-\Theta(n\log n)} + 2^{-\Theta(n)} \
		=\ 2^{-\Theta(n)}.
		\end{align*}

\bigskip 

{\bf Acknowledgment.} The authors would like to thank Wojciech Samotij for many fruitful discussions.

\bibliographystyle{abbrv}
\bibliography{resilience}
	
\appendix

\section{Proof of Hal\'asz's inequality over $\F_p$}
\label{app:halasz}
In this appendix, we prove \cref{thm:halasz-fp}. The proof follows Hal\'asz's original proof in~\cite{halasz1977estimates}. 
\begin{proof}[Proof of \cref{thm:halasz-fp}]
  Let $e_p$ be the canonical generator of the Pontryagin dual of $\F_p$, that is, the function $e_p \colon \F_p \to \C$ defined by $e_p(x) = \exp(2\pi i x / p)$. Recall the following discrete Fourier identity in $\F_p$:
  \[
    \delta_0(x) = \frac{1}{p} \sum_{r \in \F_p}e_p(rx),
  \]
  where $\delta_0(0) = 1$ and $\delta_0(x) = 0$ if $x \neq 0$. Let $\epsilon_1,\dotsc,\epsilon_n$ be i.i.d.\ Rademacher random variables. Note that for any $q \in \F_p$, 
  \begin{align*}
    \Pr\left(\sum_{j=1}^n \epsilon_j a_j = q\right) &= \E\left[\delta_0 \left(\sum_{j=1}^n \epsilon_j a_j - q\right)\right] \\
                                                    &= \E\left[ \frac{1}{p} \sum_{r\in \F_p} e_p\left(r \left(\sum_{j=1}^n \epsilon_j a_j - q\right)\right)\right] \\
                                                    & = \E\left[\frac{1}{p} \sum_{r \in \F_p} \prod_{j=1}^n e_p( \epsilon_j r a_j) e_p(-r q)\right] \\
                                                    & = \frac{1}{p} \sum_{r \in \F_p} e_p(-rq) \prod_{j=1}^n \E\big[e_p(\epsilon_j r a_j)\big].
  \end{align*}
  Since each $\epsilon_j$ is a Rademacher random variable, we have
  \[
    \E\big[e_p(\epsilon_jra_j)\big] = \exp(2\pi i ra_j/p)/2 + \exp(-2\pi i ra_j/p)/2 = \cos(2\pi r a_j/p).
  \]
  It thus follows from the triangle inequality that
  \begin{equation}
    \label{eq:Pr-cos-UB}
    \Pr\left(\sum_{j=1}^n \epsilon_j a_j = q\right) \le \frac{1}{p} \sum_{r\in \F_p} \prod_{j=1}^n \left|\cos(2 \pi r a_j/p)\right| = \frac{1}{p} \sum_{r \in \F_p} \prod_{j=1}^n \left|\cos\left(\pi r a_j/p\right) \right|,
  \end{equation}
  where the equality holds because the map $\F_p \ni r \mapsto 2r \in \F_p$ is a bijection (as $p$ is odd) and (since $x \mapsto |\cos(\pi x)|$ has period $1$ and it is therefore well defined for $x \in \R/\Z$) because $|\cos(2\pi x/p)| = |\cos(\pi (2x)/p)|$ for every $x \in \F_p$.

  Given a real number $y$, denote by $\|y\| \in [0,1/2]$ the distance between $y$ and a nearest integer. Let us record the useful inequality 
\[
  |\cos(\pi y)| \leq \exp\big(-\| y \|^2/2\big),
\]
which is valid for every real number $y$. Using this inequality to bound from above each of the $n$ terms in the right-hand side of~\cref{eq:Pr-cos-UB}, we arrive at
\begin{equation}
  \label{eqn:halasz-prelim}
  \max_{q\in \F_p}\Pr\left(\sum_{i=1}^n \epsilon_i a_i = q\right) \leq \frac{1}{p} \sum_{r\in \F_p} \exp\left(-\frac{1}{2} \sum_{i=1}^n\|r a_i/p  \|^2\right).
\end{equation}
Now, for each nonnegative real $t$, we define the following `level' set:
\[
  T_t := \left\{r \in \F_p : \sum_{j=1}^{n} \|r a_j/p  \|^2 \leq t \right\}.
\]
Since for every real $y$, we may write $e^{-y} = \int_0^\infty \mathds{1}[y \le t] e^{-t}\,dt$, then
\begin{equation}
  \label{eqn:halasz-integral}
  \sum_{r\in \F_p} \exp\left(-\frac{1}{2} \sum_{j=1}^{n} \|r a_j/p  \|^2\right) = \sum_{r \in \F_p} \int_0^\infty \mathds{1}\left[\sum_{j=1}^n \|r a_j / p \|^2 \le 2t\right] e^{-t} \, dt = \int_0^{\infty} |T_{2t}| e^{-t} \, dt.
\end{equation}
Since for every nonzero $a \in \F_p$, the map $\F_p \ni r \mapsto ra \in \F_p$ is bijective, we have
\begin{align*}
  \sum_{r \in \F_p} \sum_{j=1}^n \|r a_j/p \|^2 & = \sum_{j \in \supp(\boldsymbol{a})} \sum_{r \in \F_p} \|r a_j/p\|^2 = |\supp(\boldsymbol{a})| \sum_{r \in \F_p} \|r/p\|^2 \\
                                                & = |\supp(\boldsymbol{a})| \cdot 2\sum_{s=1}^{(p-1)/2} (s/p)^2 = |\supp(\boldsymbol{a})| \cdot \frac{p^2-1}{12p} > \frac{|\supp(\boldsymbol{a})| \cdot p}{15},
\end{align*}
where the inequality holds because $p \ge 3$ (as $p$ is an odd prime). On the other hand, it follows from the definition of $T_t$ that for every $t \ge 0$,
\[
  \sum_{r \in \F_p} \sum_{j=1}^n \|r a_j/p \|^2 \le |T_t| \cdot t + \big(p - |T_t|\big) \cdot n.
\]
This implies that $|T_t| < p$ as long as $t \le |\supp(\boldsymbol{a})|/15$.

Recall that the Cauchy--Davenport theorem states that every pair of nonempty $A, B \subseteq \F_p$ satisfies $|A+B| \ge \min\{ p, |A|+|B|-1\}$. It follows that for every positive integer $m$ and every $t \ge 0$, the iterated sumset $mT_t$ satifies $|mT_t| \ge \min\{ p, m|T_t|-m\}$. We claim that for every $m$, the iterated sumset $mT_t$ is contained in the set $T_{m^2 t}$ and thus
\[
    |T_{m^2t}| \ge \min\big\{p, m|T_t|-m\big\}.
\]
Indeed, for $r_1,\dots, r_m \in T_t$, it follows from the triangle inequality and the Cauchy--Schwarz inequality that 
\begin{align*}
  \sum_{j=1}^n \left\| \sum_{i=1}^{m} r_i a_j/ p \right\|^2 \leq \sum_{j=1}^{n} \left(\sum_{i=1}^{m} \left\|r_i a_j/p \right\|\right)^2  \leq \sum_{j=1}^{n} m \sum_{i=1}^{m} \left \|r_i a_j/p\right \|^2  \leq m^2 t.
\end{align*}
Since $|T_{m^2t}| < p$ as long as $m^2t \le |\supp(\boldsymbol{a})|/15$, we see that if $t \le 2M \le |\supp(\boldsymbol{a})|/15$, then, letting $m = \lfloor \sqrt{2M/t} \rfloor \ge 1$, we obtain
\begin{equation}
  \label{eqn:halasz-cd}
  |T_t| \le \frac{|T_{m^2t}|}{m}+1 \le \frac{\sqrt{2t} \cdot |T_{2M}|}{\sqrt{M}} + 1.
\end{equation}

We now bound the size of $T_{2M}$. First, it follows from the elementary inequality
\[
  \cos(2\pi y) \ge 1-2\pi^2\|y\|^2 \ge 1 - 20\|y\|^2,
\]
which holds for all $y\in \mathbb{R}$, that $T_{2M} \subseteq T'$, where 
\[
  T' := \left\{r \in \F_p : \sum_{j=1}^n \cos(2 \pi r a_j/p) \geq n - 40M\right\}.
\]
Second, by Markov's inequality,
\[
  |T'| \le \frac{1}{\big(n-40M\big)^{2k}} \cdot \sum_{r \in T_M'}\left(\sum_{j=1}^n\cos(2\pi ra_j/p)\right)^{2k}.
\]
Third, by our assumption that $80Mk \le n$ and since the sequence $\big(1-1/(2k)\big)^{2k}$ is increasing,
\[
  (n-40M)^{2k} = \left(1 - \frac{40M}{n}\right)^{2k} \cdot n^{2k} \ge \left(1 - \frac{1}{2k}\right)^{2k} \cdot n^{2k} \ge \frac{n^{2k}}{\sqrt{2}}
\]
Fourth, since $T' \subseteq \F_p$ and $2\cos(2\pi r a_j / p) = e_p(ra_j) + e_p(-ra_j)$, we also have
\begin{align*}
  \sum_{r \in T'} \left( \sum_{j=1}^n \cos(2 \pi r a_j / p)\right)^{2k} &\le \sum_{r \in \F_p} \left( \sum_{j=1}^n \big(e_p(ra_j) + e_p(-ra_j)\big)/2\right)^{2k}\\
                                                                          &= \frac{1}{2^{2k}} \sum_{(\sigma_1,\dotsc, \sigma_{2k})\in\{\pm 1\}^{2k}} \sum_{j_1, \dots, j_{2 k}} \sum_{r \in F_p} e_p\left( r \sum_{\ell=1}^{2 k} \sigma_\ell a_{j_\ell}\right) \\
                                                                          &= \frac{1}{2^{2k}}\sum_{(\sigma_1,\dots, \sigma_{2k})\in\{\pm 1\}^{2k}} \sum_{j_1,\dots, j_{2k}} p \cdot \delta_0\left(\sum_{\ell=1}^{2k} \sigma_\ell a_{j_\ell}\right)\\
                                                                          &= \frac{p R_k(\boldsymbol{a})}{2^{2k}}.
\end{align*}
Thus, we may conclude that
\begin{equation}
  \label{eqn:halasz-moment}
  |T_M| \le |T'| \le \frac{\sqrt{2} p R_k(\boldsymbol{a})}{2^{2k}n^{2k}}
\end{equation}

Finally, combining this with \cref{eqn:halasz-prelim,eqn:halasz-integral,eqn:halasz-cd,eqn:halasz-moment}, we get, %for any $0\leq \mu \leq 1/2$, 
\begin{align*}
  \max_{q\in\F_{p}}\Pr\left(\sum_{j=1}^n\epsilon_ja_j = q\right) & \leq \frac{1}{p}\int_{0}^{M} |T_{2t}| e^{-t}\,dt+\frac{1}{p}\int_{M}^\infty pe^{-t} \, dt\\
 & \le \frac{1}{p}\int_0^{M} \left(\frac{\sqrt{2t} \cdot |T_M|}{\sqrt{M}}+1\right)e^{-t}\,dt+e^{-M} \\
 & \le \frac{|T_{2M}|}{p\sqrt{M}} \cdot \int_0^{M} \sqrt{t}e^{-t} \, dt+ \frac{1}{p} \int_0^{M} e^{-t}\, dt +e^{-M}\\
 & \leq \frac{|T_{2M}|}{p\sqrt{M}} \cdot C' + \frac{1}{p}+e^{-M}\\
                                                                 & \le \frac{CR_{k}(\boldsymbol{a})}{2^{2k}n^{2k}\sqrt{M}} + \frac{1}{p}+e^{-M},
\end{align*}
as desired.
\end{proof}

\section{Proof of the counting theorem}
\label{sec:pf-counting-thm}

In this section, we prove~\cref{thm:counting-lemma} using an elementary double counting argument.

\begin{proof}[Proof of~\cref{thm:counting-lemma}]
  Let $\cZ$ be the set of all triples
  \[
    \left(I, \left(i_{s+1},\dots,i_{n}\right), \left(F_{j},{\boldsymbol{\epsilon}}^{j}\right)_{j=s+1}^{n} \right),
  \]
  where
  \begin{enumerate}[{label=(\roman*)}]
  \item $I \subseteq [n]$ and $|I|=s$, 
  \item $(i_{s+1},\dotsc,i_n) \in [n]^{n-s}$ is a permutation of $[n]\setminus I$,
  \item each $F_{j}:=(\ell_{j,1},\dotsc,\ell_{j,2k})$ is a sequence of $2k$ elements of $[n]$, and
  \item $\boldsymbol{\epsilon}^j\in \{\pm 1\}^{2k}$ for each $j$,
  \end{enumerate}
  that satisfy the following conditions for each $j$:
  \begin{enumerate}[{label=(\alph*)}]
  \item
    \label{item:Z-condition-1}
    $\ell_{j,2k}= i_{j}$ and
  \item
    \label{item:Z-condition-2}
    $(\ell_{j,1},\dotsc,\ell_{j,2k-1}) \in \big(I\cup \{i_{s+1},\dotsc,i_{j-1}\}\big)^{2k-1}$.
  \end{enumerate}

  \begin{claim}
    The number of triples in $\cZ$ is at most $(s/n)^{2k-1} \cdot \big(2^{n-s} n! / s!\big)^{2k}$.
  \end{claim}
  \begin{proof}
    One can construct any such triple as follows. First, choose an $s$-element subset of $[n]$ to serve as $I$. Second, considering all $j \in \{s+1, \dotsc, n\}$ one by one in increasing order, choose: one of  the $n-j+1$ remaining elements of $[n] \setminus I$ to serve as $i_j$; one of the $2^{2k}$ possible sign patterns to serve as $\boldsymbol{\epsilon}^{j}$; and one of the $(j-1)^{2k-1}$ sequences of $2k-1$ elements of $I \cup \{i_{s+1},\dots,i_{j-1}\}$ to serve as $(\ell_{j,1},\dotsc,\ell_{j,2k-1})$. Therefore,
    \begin{align*}
      |\cZ| & \le \binom{n}{s} \cdot \prod_{j=s+1}^n \left((n-j+1) \cdot 2^{2k} \cdot (j-1)^{2k-1}\right) \\
            & = \frac{n!}{s!(n-s)!} \cdot (n-s)! \cdot 2^{2k(n-s)} \cdot \left(\frac{(n-1)!}{(s-1)!}\right)^{2k-1} = \left(\frac{s}{n}\right)^{2k-1} \cdot \left(2^{n-s} \cdot \frac{n!}{s!}\right)^{2k}.\qedhere
    \end{align*}
  \end{proof}
  
  We call $\boldsymbol{a} = (a_1, \dotsc, a_n) \in \F_p^n$ \emph{compatible} with a triple from $\cZ$ if for every $j \in \{s+1, \dotsc, n\}$,
  \begin{equation}
    \label{eq:compatibility}
    \sum_{i=1}^{2k}\boldsymbol\epsilon^{j}_ia_{\ell_{j,i}}=0.
  \end{equation}
  
  \begin{claim}
    Each triple from $\cZ$ is compatible with at most $p^s$ sequences $\boldsymbol{a} \in \F_p^n$.
  \end{claim}
  \begin{proof}
    Using~\ref{item:Z-condition-1}, we may rewrite~\cref{eq:compatibility} as
    \[
      \boldsymbol{\epsilon}^{j}_{2k}a_{i_{j}} = -\sum_{i=1}^{2k-1}\boldsymbol\epsilon^{j}_ia_{\ell_{j,i}}.
    \]
    It follows from~\ref{item:Z-condition-2} that once a triple from $\cZ$ is fixed, the right-hand side above depends only on those coordinates of the vector $\boldsymbol{a}$ that are indexed by $i \in I \cup \{i_{s+1}, \dotsc, i_{j-1}\}$. In particular, for each of the $p^s$ possible values of $(a_i)_{i \in I}$, there is exactly one way to extend it to a sequence $\boldsymbol{a} \in \F_p^n$ that satisfies~\cref{eq:compatibility} for every $j$.
  \end{proof}

  \begin{claim}
    Each sequence $\boldsymbol{a} \in\Bad_{k,s, \ge t}^\alpha$ is compatible with at least
    \[
     \left(\frac{2^{n-s}n!}{s!}\right)^{2k} \cdot \left(\frac{\alpha t}{p}\right)^{n-s}
    \]
    triples from $\cZ$.
  \end{claim}
  \begin{proof}
    Given any such $\boldsymbol{a}$, we may construct a compatible triple from $\cZ$ as follows. Considering all $j \in \{n, \dotsc, s+1\}$ one by one in decreasing order, we do the following. First, we find an arbitrary solution to
    \begin{equation}
      \label{eq:a-ell-solution}
      \pm a_{\ell_1} \pm a_{\ell_2} \pm \dotsb \pm a_{\ell_{2k}} = 0
    \end{equation}
    such that $\ell_1, \dotsc, \ell_{2k} \in [n]\setminus \{i_{n},\dots,i_{j+1}\}$ and such that $\ell_{2k}$ is a non-repeated index (i.e., such that $\ell_{2k} \neq \ell_i$ for all $i \in [2k-1]$). Given any such solution, we let $\ell_{2k}$ serve as $i_j$, we let the sequence $(\ell_1, \dotsc, \ell_{2k})$ serve as $F_j$, and we let $\boldsymbol{\epsilon}^j$ be the corresponding sequence of signs (so that~\cref{eq:compatibility} holds). The assumption that $\boldsymbol{a} \in \Bad_{k,s,\geq t}^\alpha(n)$ guarantees that there are at least  $t \cdot \frac{2^{2k}\cdot (n-j+1)^{2k}}{p}$ many solutions to~\cref{eq:a-ell-solution}, each of which has at least $2\alpha k$ nonrepeated indices. Since the set of all such solutions is closed under every permutation of the $\ell_i$s (and the respective signs), $\ell_{2k}$ is a non-repeated index in at least an $\alpha$-proportion of them. Finally, we let $I = [n] \setminus \{i_n, \dotsc, i_{s+1}\}$. Since different sequences of solutions lead to different triples, it follows that the number $Z$ of compatible triples satisfies
    \[
      Z \ge \prod_{j = s+1}^{n} \left(\alpha t \cdot \frac{2^{2k} \cdot (n-j+1)^{2k}}{p}\right) = \left(\frac{2^{n-s}n!}{s!}\right)^{2k}\cdot \left(\frac{\alpha t}{p}\right)^{n-s}.\qedhere
    \]
  \end{proof}

  Counting the number $P$ of pairs of $\boldsymbol{a} \in \Bad_{k, s, \ge t}^\alpha(n)$ and a compatible triple from $\cZ$, we have
  \[
    |\Bad_{k, s, \ge t}^\alpha(n)| \cdot \left(\frac{2^{n-s}n!}{s!}\right)^{2k} \cdot \left(\frac{\alpha t}{p}\right)^{n-s}\le P \le |\cZ| \cdot p^s \le \left(\frac{s}{n}\right)^{2k-1} \cdot \left(\frac{2^{n-s}n!}{s!}\right)^{2k} \cdot p^s,
  \]
  which yields the desired upper bound on $|\Bad_{k, s, \ge t}^\alpha(n)|$.
\end{proof}	
	
\end{document}